\documentclass[12pt]{article}
\title{A note on \textit{fsg} groups in $p$-adically closed fields}
\author{Will Johnson}

\usepackage{amsmath, amssymb, amsthm}    	
\usepackage{fullpage} 	
\usepackage{amscd}
\usepackage{hyperref}
\usepackage[all]{xy}
\usepackage{centernot}

\DeclareMathOperator*{\forkindep}{\raise0.2ex\hbox{\ooalign{\hidewidth$\vert$\hidewidth\cr\raise-0.9ex\hbox{$\smile$}}}}

\newcommand{\Av}{\operatorname{Av}}

\newcommand{\Aut}{\operatorname{Aut}}

\newtheorem{theorem}{Theorem}[section] 
\newtheorem{lemma}[theorem]{Lemma}

\newtheorem{corollary}[theorem]{Corollary}
\newtheorem{fact}[theorem]{Fact}

\newtheorem{proposition}[theorem]{Proposition}
\newtheorem{proposition-eh}[theorem]{Proposition(?)}
\newtheorem*{theorem-star}{Theorem}
\newtheorem*{conjecture-star}{Conjecture}
\newtheorem*{lemma-star}{Lemma}
\newtheorem{claim}[theorem]{Claim}

\theoremstyle{definition}
\newtheorem{definition}[theorem]{Definition}

\newtheorem{remark}[theorem]{Remark}

\theoremstyle{remark}

\newtheorem*{acknowledgment}{Acknowledgments}

\newcommand{\Qq}{\mathbb{Q}}

\newcommand{\Rr}{\mathbb{R}}

\newcommand{\Mm}{\mathbb{M}}

\newcommand{\Oo}{\mathcal{O}}

\newenvironment{claimproof}[1][\proofname]
               {
                 \proof[#1]
                 
               }
               {
                 \endproof
               }

\begin{document}
\maketitle

\begin{abstract}
  Let $G$ be a definable group in a $p$-adically closed field $M$.  We
  show that $G$ has finitely satisfiable generics (\textit{fsg}) if and
  only if $G$ is definably compact.  The case $M = \Qq_p$ was
  previously proved by Onshuus and Pillay.
\end{abstract}

\section{Introduction}
Work in a monster model $\Mm$ of some theory.  Let $G$ be a definable
group.  Say that a definable set $X \subseteq G$ is \emph{left
generic} or \emph{right generic} if $G$ can be covered by finitely
many left translates or right translates of $X$, respectively.  A
definable group $G$ is said to have \emph{finitely satisfiable
generics} (\textit{fsg}) if there is a small model $M_0$ and a global
type $p \in S_G(\Mm)$ such that every left translate $g \cdot p$ is
finitely satisfiable in $M_0$.  This notion is due to Hrushovski,
Peterzil, and Pillay \cite{goo}, who prove the following facts:
\begin{fact}[{\cite[Proposition~4.2]{goo}}]\label{hpp-fact}
  Suppose $G$ has \textit{fsg}, witnessed by $p$ and $M_0$.
  \begin{enumerate}
  \item A definable set $X \subseteq G$ is left generic iff it is
    right generic.
  \item Non-generic sets form an ideal: if $X \cup Y$ is generic, then
    $X$ is generic or $Y$ is generic.
  \item A definable set $X$ is generic if and only if every left
    translate of $X$ intersects $G(M_0)$.
  \end{enumerate}
\end{fact}
If $G$ is a group definable in a nice o-minimal structure, then $G$
has \textit{fsg} if and only if $G$ is definably compact
\cite[Remark~5.3]{udi-anand}.  Our main theorem is an analogue for
$p$-adically closed fields:
\begin{theorem} \label{main-thm}
  Let $M$ be a $p$-adically closed field and $G$ be an $M$-definable
  group.  Then $G$ is definably compact if and only if $G$ has
  \textit{fsg}.
\end{theorem}
The case $M = \Qq_p$ was proved by Onshuus and Pillay
\cite[Corollary~2.3]{O-P}.

\subsection{Notation}
We denote the theory of $p$-adically closed fields by $p$CF.  In a
$p$-adically closed field $M$, we will let $\Gamma$ or $\Gamma(M)$
denote the value group, and $\Oo$ or $\Oo(M)$ denote the valuation
ring.  The valuation will be denoted $v(x)$, and written additively,
so $v(xy) = v(x) + v(y)$ and $v(x+y) \ge \min(v(x),v(y))$.  If $D$ is
a definable set, we write $S_D(M)$ for the set of complete types over
$M$ concentrating on $D$.

\subsection{Outline}
In Section~\ref{sec:tops} we review the notion of definable
compactness in definable groups in $p$-adically closed fields.  In
Section~\ref{sec:easy} we give the easy direction of
Theorem~\ref{main-thm}: if $G$ has \textit{fsg} then $G$ is definably compact.  In Section~\ref{sec:definable} we show that
definable compactness is a definable property---it varies definably in
a definable family of definable groups.  This ensures that every
definably compact group is part of a 0-definable family of definably
compact groups.  In Section~\ref{sec:hard} we use the VC-theorem to
show that \textit{fsg} is witnessed in a very uniform manner for
definably compact groups over $\Qq_p$.  In
Section~\ref{sec:conclusion} we show that definable compactness
implies \textit{fsg} by using Section~\ref{sec:definable} to transfer
the ``uniform \textit{fsg}'' of Section~\ref{sec:hard} from $\Qq_p$ to
its elementary extensions.

\section{Definable compactness in $p$CF} \label{sec:tops}
If $X$ is a definable set and $\tau$ is a topology on $X$, we say that
$\tau$ is \emph{definable} if there is a definable family $\{B_i\}_{i
  \in I}$ such that $\{B_i : i \in I\}$ is a basis for the topology.
Two examples of definable topologies are the order topology on an
o-minimal structure and the valuation topology on a $p$-adically
closed field.  A \emph{definable topological space} is a definable set
with a definable topology.

We will use the following abstract notion of definable compactness,
which works well in $p$-adically closed fields and o-minimal
structures:
\begin{definition}[{\cite{fornasiero, wj-o-minimal}}]\label{f-j-def}
  A definable topological space $X$ is \emph{definably compact} if the
  following holds: if $\{F_i\}_{i \in I}$ is a definable family of
  non-empty closed subsets of $X$, and $\{F_i : i \in I\}$ is
  downwards-directed, then the intersection $\bigcap_{i \in I} F_i$ is
  non-empty.

  A definable subset $Y \subseteq X$ is definably compact if it is
  compact with the induced subspace topology, which is definable.
\end{definition}

Definable compactness has many properties analogous to compactness
\cite[Section~3.1]{wj-o-minimal}.  We will need the following two
trivial observations:
\begin{fact}\label{k-facts}~
\begin{enumerate}
\item If $f : X \to Y$ is a definable continuous function between
  definable topological spaces, and $D \subseteq X$ is definably
  compact, then the image $f(D) \subseteq Y$ is definably compact.
\item If $X$ is a definable topological space and $D_1, D_2 \subseteq
  X$ are definably compact, then $D_1 \cup D_2$ is definably compact.
\end{enumerate}  
\end{fact}

\subsection{Definable manifolds}
Work in a $p$-adically closed field $M$.  A \emph{definable manifold}
is a Hausdorff definable topological space $X$ covered by finitely
many definable opens $U_1, \ldots, U_k$, such that each $U_i$ is
definably homeomorphic to an open subset of $M^n$.  Definable
manifolds arise naturally in the study of definable groups by the
following theorem of Pillay:
\begin{fact}[{\cite{Pillay-G-in-p}}] \label{pillay-manifold}
  If $G$ is a definable group (in a $p$-adically closed field), then
  there is a unique definable topology $\tau$ on $G$ such that
  \begin{itemize}
  \item The group operations are continuous.
  \item $(G,\tau)$ is a definable manifold.
  \end{itemize}
\end{fact}
On definable manifolds, we can give a concrete characterization of
definable compactness.  First consider the case $M^n$:
\begin{fact}[{\cite[Lemmas~2.4, 2.5]{johnson-yao}}]\label{base-situation}
  A definable set $D \subseteq M^n$ is definably compact iff it is
  closed and bounded.
\end{fact}
Here, a set $D$ is ``bounded'' if there is $N \in \Gamma$ such that
$v(x_i) > N$ for all $x_i \in D$.

This can be generalized to other definable manifolds using the
following notion:
\begin{definition}
  Let $X$ be a definable manifold.  A \emph{$\Gamma$-exhaustion} is a
  definable family $\{W_\gamma\}_{\gamma \in \Gamma}$ such that
  \begin{enumerate}
  \item Each $W_\gamma$ is an open, definably compact subset of $X$.
  \item $\gamma \le \gamma' \implies W_\gamma \subseteq W_{\gamma'}$.
  \item $X = \bigcup_{\gamma \in \Gamma}$.
  \end{enumerate}
\end{definition}
For example, in $M^n$, if $B_\gamma(0)$ denotes the ball of radius
$\gamma$ around $0 \in M^n$, then $\{B_{-\gamma}(0)\}_{\gamma \in
  \Gamma}$ is a $\Gamma$-exhaustion.
\begin{fact}[{\cite[Remark~2.8]{johnson-yao}}] \label{gamma-tech}
  Let $X$ be a definable manifold.
  \begin{enumerate}
  \item \label{gt-1} There is at least one $\Gamma$-exhaustion on $X$.
  \item \label{gt-2} Suppose we write $X$ as a finite union $U_1 \cup \cdots \cup
    U_k$ of definable open sets.  Suppose that for $i < k$, the family
    $\{W^i_\gamma\}$ is a $\Gamma$-exhaustion of $U_i$.  Let
    $V_\gamma = \bigcup_{i = 1}^k W^i_\gamma$.  Then
    $\{V_\gamma\}$ is a $\Gamma$-exhaustion of $X$.
  \end{enumerate}
\end{fact}
We can then characterize definable compactness as follows:
\begin{fact} \label{char}
  Let $X$ be a definable manifold and $\{W_\gamma\}$ be a
  $\Gamma$-exhaustion.  Let $D \subseteq X$ be a definable set.  The
  following are equivalent:
  \begin{enumerate}
  \item \label{ch-1} $D$ is definably compact.
  \item \label{ch-2} $D$ is closed, and $D$ is \emph{bounded}, in the sense that $D
    \subseteq W_\gamma$ for some $\gamma$.
  \item \label{ch-3} For any definable continuous function $f : \Oo \setminus \{0\}
    \to D$, there is a point $p \in D$ which is a \emph{cluster point}
    of $f$ at 0, in the sense that for any neighborhood $U$ of $p$ and
    $V$ of 0, there is $x \in V \setminus \{0\}$ such that $f(x) \in
    U$.
  \item \label{ch-4} If $r$ is a 1-dimensional definable type concentrating on $D$,
    then there is a point $p \in D$ such that $r$ \emph{specializes
    to} $p$, in the sense that for any definable neighborhood $U$ of
    $p$, the type $r$ concentrates on $U$.
  \end{enumerate}
\end{fact}
The equivalence of (\ref{ch-1}) and (\ref{ch-2}) follows from
Proposition~2.9 and Fact~2.2(4--5) in \cite{johnson-yao}.  The
equivalence of (\ref{ch-1}) and (\ref{ch-3}) is
\cite[Proposition~2.15]{johnson-yao}.  The equivalence of (\ref{ch-1})
and (\ref{ch-4}) is \cite[Proposition~2.24]{johnson-yao}.
\begin{fact}[{\cite[Remark~2.12]{johnson-yao}}] \label{standard-compare}
  Let $X$ be a $\Qq_p$-definable manifold and $Y \subseteq X$ be
  $\Qq_p$-definable.  Then $Y$ is definably compact if and only if
  $Y(\Qq_p)$ is compact as a subset of the $p$-adic manifold
  $X(\Qq_p)$.
\end{fact}
\begin{remark}
  Suppose $M \preceq N$ are two models of $p$CF, and $X = X(M)$ is a
  definable topological space in $M$.  Let $X(N)$ denote the
  associated definable topological space in $N$.  One can easily show
  from Definition~\ref{f-j-def} that $X(M)$ is definably compact if
  and only if $X(N)$ is definably compact.  In other words,
  ``definable compactness'' is invariant under elmentary extensions.

  In particular, if $G = G(M)$ is a definable group in $M$, then
  $G(M)$ is definably compact in $M$ iff $G(N)$ is definably compact
  in $N$.
\end{remark}
Consequently, we can move to a monster model without changing whether
a definable group $G$ is definably compact.
\section{\textit{fsg} implies definable compactness} \label{sec:easy}
Work in a monster model $\Mm \models p$CF.
\begin{proposition} \label{easy-dir}
  Let $G$ be a definable group.  If $G$ has \textit{fsg}, then $G$ is
  definably compact.
\end{proposition}
\begin{proof}
  Take $p \in S_G(\Mm)$ and a small model $M_0 \preceq \Mm$ witnessing
  \textit{fsg}.  Take a $\Gamma$-exhaustion $\{W_\gamma\}_{\gamma \in
    \Gamma}$.  By definition, $G = \bigcup_{\gamma \in \Gamma}
  W_\gamma$.  The set $G(M_0)$ is small, so by saturation there is
  $\gamma \in \Gamma = \Gamma(\Mm)$ such that $G(M_0) \subseteq
  W_{\gamma}$.  Let $D = W_{\gamma}$ and $D'$ be the complement $G
  \setminus D$.  Then $D' \cap G(M_0) = \varnothing$.  By
  Fact~\ref{hpp-fact}, the definable set $D'$ is not generic, and
  therefore $D$ is generic.  Then finitely many left translates of $D$
  cover $G$:
  \begin{equation*}
    G = a_1 \cdot D \cup \cdots \cup a_k \cdot D.
  \end{equation*}
  The maps $x \mapsto a_i \cdot x$ are continuous, so by
  Fact~\ref{k-facts}, $G$ is definably compact.
\end{proof}

\section{Definable compactness is definable in families} \label{sec:definable}
Work in a monster model $\Mm \models p$CF.
\begin{proposition} \label{prop-annoying}
  Let $\{G_i\}_{i \in I}$ be a definable family of definable groups.
  Then the set
  \begin{equation*}
    \{i \in I : G_i \text{ is definably compact}\}
  \end{equation*}
  is definable.
\end{proposition}
If you stare carefully at Fact~\ref{base-situation},
Fact~\ref{gamma-tech}, and the equivalence between (\ref{ch-1}) and
(\ref{ch-2}) in Fact~\ref{char}, you can convince yourself that this
is automatically true.  But we include the details for completeness.
\begin{proof}
  Let $n$ be the dimension of $G$.  Take a small model $M_0$ defining
  the family $\{G_i\}$.  Let $X = \{i \in I : G_i \text{ is definably
    compact}\}$.  It suffices to show that both $X$ and $I \setminus
  X$ are $\vee$-definable, i.e., unions of $M_0$-definable sets.  We
  consider $X$; the proof for $I \setminus X$ is similar.

  Take some $i_0 \in X$, so that $G_{i_0}$ is definably compact.  Fix
  the following data:
  \begin{enumerate}
  \item \label{dat-1} Finitely many open definable sets $U_1, \ldots, U_k$ covering
    $G_{i_0}$.
  \item \label{dat-2} Open definable sets $V_j \subseteq \Mm^n$ for $j \le k$ and
    definable homeomorphisms $h_j : U_j \to V_j$.
  \item \label{dat-3} For each $j$, a $\Gamma$-exhaustion
    $\{W^j_{\gamma}\}_{\gamma \in \Gamma}$ of $U_j$.
  \end{enumerate}
  These exist by Facts~\ref{pillay-manifold} and
  \ref{gamma-tech}(\ref{gt-2}).  Take a finite tuple $b_0 \in
  \Mm^{\ell}$ over which the data in (\ref{dat-1})--(\ref{dat-3}) are
  definable.  We can define some $U_1^b$ for $b \in \Mm^{\ell}$ such
  that $U_1^b$ depends 0-definably on $b$, and $U_1 = U_1^{b_0}$.
  Define $U_j^b, V_j^b, h^b_j$ and $\{W^{j,b}_\gamma\}_{\gamma \in
    \Gamma}$ for $1 \le j \le k$ in a similar fashion.

  There is an $\mathcal{L}_{M_0}$ formula $\phi(x,y)$ such that
  $\phi(i,b)$ holds if and only if the following eight conditions
  hold:
  \begin{itemize}
  \item $i \in I$.  (This can be expressed because $I$ is $M_0$-definable.) 
  \item Each $U^b_j$ is a subset of $G_i$, and $G_i = \bigcup_{j =
    1}^k U^b_j$.
  \item Each $V^b_j$ is an open subset of $\Mm^n$.
  \item Each $h^b_j$ is a bijection from $U^b_j$ to $V^b_j$.
  \item The collection of $h^b_j : U^b_j \to V^b_j$ for $j = 1,
    \ldots, k$ is an atlas making $G_i$ into a Hausdorff definable
    manifold.
  \item The group operations on $G_i$ are continuous with respect to
    the definable manifold structure.
  \item Each $\{W^{j,b}_\gamma\}_{\gamma \in \Gamma}$ is a
    $\Gamma$-exhaustion on $U^b_j$.  (In order to express that
    $W^{j,b}_\gamma$ is definably compact, use the homeomorphism
    $h^b_j : U^b_j \to V^b_j$.  Fact~\ref{base-situation} shows how to
    express definable compactness for definable subsets of $V^b_j$.)
  \item If we let $\tilde{W}^b_\gamma = \bigcup_{j = 1}^k
    W^{j,b}_\gamma$, then there is some\footnote{For $I \setminus
    X$, change ``some'' to ``no''.} $\gamma$ such that
    $\tilde{W}^b_\gamma = G_i$.
  \end{itemize}
  \begin{claim}
    $\phi(i_0,b_0)$ holds
  \end{claim}
  \begin{claimproof}
    All the bullet points are clear except the last one.  Note that
    $\{\tilde{W}^{b_0}_\gamma\}$ is a $\Gamma$-exhaustion of $G_{i_0}$
    by Fact~\ref{gamma-tech}(\ref{gt-1}).  Then there is\footnote{For
    $I \setminus X$, insert ``not''.}  $\gamma$ such that $G_{i_0}
    \subseteq \tilde{W}^{b_0}_\gamma$, by the equivalence of parts
    (\ref{ch-1}) and (\ref{ch-2}) in Fact~\ref{char}.
  \end{claimproof}
  \begin{claim}
    If $\phi(i,b)$ holds, then $G_i$ is definably compact.
  \end{claim}
  \begin{claimproof}
    The definition of $\phi$ ensures that $i \in I$ and the sets
    $U^b_1, \ldots, U^b_k$ are an open cover of $G_i$.  The family
    $\{\tilde{W}^b_\gamma\}_{\gamma \in \Gamma}$ appearing in the eighth bullet
    point is a $\Gamma$-exhaustion of $G_i$, by
    Fact~\ref{gamma-tech}(\ref{gt-2}).  The eighth bullet point then
    ensures that $G_i$ is\footnote{For $I \setminus X$, insert
    ``not''.} definably compact.
  \end{claimproof}
  Combining the two claims, we see that
  \begin{equation*}
    i_0 \in \{i \in I \mid \exists b \in \Mm^\ell : \phi(i,b)\} \subseteq X.
  \end{equation*}
  So there is an $M_0$-definable set containing $i_0$ and contained in
  $X$.  As $i_0$ is an arbitrary element of $X$, it follows that $X$
  is a union of $M_0$-definable sets.

  A nearly identical argument shows that $I \setminus X$ is a union of
  $M_0$-definable sets.  By saturation, $X$ is definable.
\end{proof}
\begin{corollary} \label{duh-cor}
  Let $G$ be a definably compact definable group.  Then there is a
  0-definable family $\{G_i\}_{i \in I}$ such that $G = G_i$ for some
  $i$, and every $G_i$ is definably compact.
\end{corollary}
\begin{proof}
  We can always find some 0-definable family of definable groups
  $\{G_i\}_{i \in J}$ containing $G$.  Let $I$ be the set of $i \in J$
  such that $G_i$ is definably compact.  Then $I$ is
  $\Aut(\Mm/\varnothing)$-invariant, and definable by
  Proposition~\ref{prop-annoying}.  Consequently, $I$ is 0-definable.
  Then the family $\{G_i\}_{i \in I}$ is 0-definable and contains $G$.
\end{proof}

\section{Uniform witnesses to \textit{fsg}} \label{sec:hard}
If $\phi(x,y)$ is an $\mathcal{L}_{rings}$-formula, then the
\emph{VC-dimension} of $\phi(x,y)$ will denote the VC-dimension of
$\phi(x,y)$ in $p$CF, i.e., the largest $n$ such that there is $M
\models p$CF and a set $\{a_1, \ldots, a_n\} \in M^{|x|}$ shattered by
$\phi$, meaning that for any $S \subseteq \{a_1,\ldots,a_n\}$ there is
$b \in M^{|y|}$ such that
\begin{equation*}
  S = \{a_1,\ldots,a_n\} \cap \phi(M,b).
\end{equation*}
The VC-dimension of $\phi$ is always finite, because $p$CF is NIP.

Work in the standard model $\Qq_p \models p$CF.  For definable groups
$G$, compactness is equivalent to definable compactness
(Fact~\ref{standard-compare}).  Any compact definable group $G$ has a
Haar measure $\mu = \mu_G$, which we normalize to make $\mu(G) = 1$.
Any definable set is Borel, hence measurable.
\begin{proposition}\label{main-target}
  Let $G$ be a compact definable group.  Let $\mu$ be normalized Haar
  measure on $G$.  Let $\phi(x,y)$ be a formula.  Let $\varepsilon >
  0$.  There is $\{a_1,\ldots,a_N\} \in G$ such that for any
  $\phi$-set $D$ contained in $G$,
  \begin{equation*}
    \mu(D) > \varepsilon \implies D \cap \{a_1,\ldots,a_N\} \ne \varnothing.
  \end{equation*}
  Moreover, $N$ can be chosen to depend only on the VC-dimension of
  $\phi$ and on $\varepsilon$, \emph{not on $G$}.
\end{proposition}
This is essentially a direct consequence of the VC-theorem, but we
include the details for lack of a reference that proves the exact
statement we want.  We will closely follow the argument from
\cite[Proposition~7.26]{NIPguide}.
\begin{proof}
  We will need some notation from \cite{NIPguide}.  If $a_1, \ldots,
  a_n \in G$ and $S \subseteq G$, then
  \begin{equation*}
    \Av(a_1,\ldots,a_n;S) = \frac{|\{i : a_i \in S\}|}{n}
  \end{equation*}
  If $a_1,\ldots,a_n,b_1,\ldots,b_n \in G$, then
  \begin{equation*}
    f_n(a_1,\ldots,a_n;b_1,\ldots,b_n) = \sup_{c \in \Qq_p^{|y|}} |\Av(a_1,\ldots,a_n;\phi(\Qq_p;c)) - \Av(b_1,\ldots,b_n;\phi(\Qq_p;c))|.
  \end{equation*}
  The following is trivial:
  \begin{claim} \label{first-claim}
    Suppose $D$ is a $\phi$-set such that $D \cap \{a_1,\ldots,a_n\} = \varnothing \ne D \cap \{b_1,\ldots,b_n\}$.  Then
    \begin{equation*}
      f_n(a_1,\ldots,a_n;b_1,\ldots,b_n) \ge |\Av(a_1,\ldots,a_n;D) - \Av(b_1,\ldots,b_n;D)| = \Av(b_1,\ldots,b_n;D).
    \end{equation*}
  \end{claim}
  Let $k$ be the VC-dimension of $\phi$.  By the Sauer-Shelah Lemma
  \cite[Lemma~6.4]{NIPguide}, the shatter function $\pi_{\phi}(n)$ is
  bounded by $\sum_{i = 0}^k \binom{n}{i}$.  (See \cite[Section
    6.1]{NIPguide} for the definition of the shatter function.)  Let
  $\delta = \varepsilon/2$.  Let $N$ be large enough that $1/(N
  \varepsilon^2) < 1 - \delta$ and
  \begin{equation*}
    4 \left(\sum_{i = 0}^k \binom{N}{i}\right) \exp \left( - \frac{N \delta^2}{8} \right) < \delta.
  \end{equation*}
  Note that $\delta$ and $N$ can be chosen to depend only on $k$ and
  $\varepsilon$.
  
  Applying \cite[Lemma 7.24]{NIPguide} to normalized Haar measure on
  $G^{2N}$, we see the following:
  \begin{claim}\label{second-claim}
    If $a_1, \ldots, a_N, b_1, \ldots, b_N$ are chosen randomly in $G$, then
    \begin{equation*}
      Prob(f_N(\bar{a},\bar{b}) > \delta) \le 4 \pi_\phi(N) \exp\left(-\frac{N \delta^2}{8}\right) < \delta.
    \end{equation*}
  \end{claim}
  By Fubini's theorem, we can fix some $a_1,\ldots,a_N$ such that the following holds:
  \begin{claim}\label{third-claim}
    If $b_1, \ldots, b_N$ are chosen randomly in $G$, then
    \begin{equation*}
      Prob(f_N(\bar{a},\bar{b}) > \delta) < \delta.
    \end{equation*}
  \end{claim}
  We claim that $a_1,\ldots,a_N$ have the desired property.
  Otherwise, there is some $\phi$-set $D$ such that $\mu(D) >
  \varepsilon$ but $D \cap \{a_1,\ldots,a_N\} = 0$.
  Combining~\ref{first-claim} and \ref{third-claim}, we see
  \begin{claim}\label{fourth-claim}
    If $b_1,\ldots,b_N$ are chosen randomly in $G$, then
    \begin{equation*}
      Prob(\Av(b_1,\ldots,b_N;D) > \delta) < \delta,
    \end{equation*}
  \end{claim}
  Since $\delta = \varepsilon/2$ and $\mu(D) > \varepsilon$, the
  following implication holds for any $b_i$:
  \begin{equation*}
    \Av(b_1,\ldots,b_N;D) \le \delta \implies |\Av(b_1,\ldots,b_N;D) - \mu(D)| \ge \varepsilon/2.
  \end{equation*}
  The weak law of large numbers \cite[Proposition~B.4]{NIPguide} shows that for random $\bar{b} \in G^N$,
  \begin{equation*}
    Prob(\Av(\bar{b};D) \le \delta) \le Prob(|\Av(\bar{b};D) - \mu(D)| \ge \varepsilon/2) \le \frac{1}{N \varepsilon^2} < 1 - \delta.
  \end{equation*}
  The event $\Av(\bar{b};D) \le \delta$ has probability less than $1 -
  \delta$, and the event $\Av(\bar{b};D) > \delta$ has probability
  less than $\delta$ (by Claim~\ref{fourth-claim}).  This is absurd.
\end{proof}
\begin{remark}\label{notation}
  We will let $N_{k,\varepsilon}$ denote the $N$ in
  Proposition~\ref{main-target} that works uniformly across all
  $\phi$-sets $D \subseteq G$ with $\mu_G(D) > \varepsilon$ and $\phi$
  of VC-dimension $k$.
\end{remark}
\begin{remark}
  Proposition~\ref{main-target} can also be seen using facts about
  generically stable measures.  Embed $\Qq_p$ into a monster model
  $\Mm$.  By \cite[Theorem~6.3]{standard-measures}, the Haar measure
  $\mu$ on $G(\Qq_p)$ is \emph{smooth}, meaning that it has a unique
  extension to a global Keisler measure $\tilde{\mu}$ on $G(\Mm)$.  By
  \cite[Lemma~7.17(i)]{NIPguide}, $\tilde{\mu}$ is generically stable.
  By \cite[Theorem~7.29(ii)]{NIPguide}, for any formula $\phi$ and any
  $\varepsilon > 0$ there is $\{a_1,\ldots,a_N\} \in G(\Qq_p)$ such
  that for any $\Qq_p$-definable $\phi$-set $D \subseteq G$, we have
  \begin{equation*}
    \left|\tilde{\mu}(D) - \frac{|D \cap \{a_1,\ldots,a_n\}|}{n}\right| < \varepsilon.
  \end{equation*}
  This implies the conclusion of Proposition~\ref{main-target}.
  Tracing through the proofs, one can see that $n$ depends only on
  $\varepsilon$ and the VC-dimension of $\phi$.
\end{remark}

\section{Definable compactness implies \textit{fsg}} \label{sec:conclusion}

Let $\Qq_p^\dag$ be the expansion of $(\Qq_p,+,\cdot)$ by the following data:
\begin{enumerate}
\item A new sort $\Rr$ with its full field structure.
\item For every 0-definable family $\mathcal{F} = \{G_i\}_{i \in I}$
  of compact definable groups and every $\mathcal{L}_{rings}$-formula
  $\phi(x,y)$ a function $f_{\mathcal{F},\phi} : I \times \Qq_p^{|y|}
  \to \Rr$ defined by
  \begin{equation*}
    f_{\mathcal{F},\phi}(i,b) = \mu_{G_i}(\phi(\Qq_p;b) \cap G_i),
  \end{equation*}
  where $\mu_G$ denotes normalized Haar measure on $G_i$.
\end{enumerate}
Let $\Mm^\dag = (\Mm,\Rr^*)$ be a monster model of $\Qq_p^\dag$, and let $\Mm$ be
the reduct to $\mathcal{L}_{rings}$ (discarding the new sort $\Rr^*$).  Then
$\Mm$ is a monster model of $p$CF.
\begin{lemma} \label{ridiculemma}
  Let $\mathcal{F} = \{G_i\}_{i \in I}$ and $\mathcal{F}' =
  \{G'_i\}_{i \in J}$ be two 0-definable families of definably compact
  groups.  Let $\phi(x;y)$ and $\psi(x;z)$ be two $\mathcal{L}$-formulas.

  Suppose $i, j \in I(\Mm)$ are such that $G_i(\Mm) = G'_j(\Mm) =: G$
  and $G \cap \phi(\Mm;b) = G \cap \psi(\Mm;c)$.  Then
  \begin{equation}
    f_{\mathcal{F},\phi}(i,b) = f_{\mathcal{F}',\psi}(j,c). \label{stupid-identity}
  \end{equation}
\end{lemma}
\begin{proof}
  For fixed $\mathcal{F}, \mathcal{F}', \phi, \psi$, the second
  paragraph of the lemma can be expressed as a single first-order
  sentence in the language of $\Mm^\dag$ and $\Qq_p^\dag$.  The
  sentence holds in $\Qq_p^\dag$ by construction, so it holds in the
  elementary extension $\Mm^\dag$.
\end{proof}
\begin{definition}
  Suppose that $G$ is a definably compact group in $\Mm$ and $D
  \subseteq G$ is a definable subset.  Let $\mathcal{F} = \{G_i\}_{i
    \in I}$ be a 0-definable family containing $G$.  (This exists by
  Corollary~\ref{duh-cor}.)  Let $i \in I$ be such that $G =
  G_{i}$.  Let $D = G \cap \phi(\Mm;b)$ for some
  $\mathcal{L}_{rings}$-formula $\phi$ and parameter $b$.  We define
  $\mu_G(D) \in \Rr^*$ by
  \begin{equation*}
    \mu_G(D) = f_{\mathcal{F},\phi}(i,b).
  \end{equation*}
  This is well-defined, independent of the choice of $\mathcal{F}, i,
  \phi, b$, by Lemma~\ref{ridiculemma}.
\end{definition}
\begin{proposition}\label{keisler-measure}
  Let $G$ be a definably compact group in $\Mm$, and let $D, D'$ be
  definable subsets of $G$.
  \begin{enumerate}
  \item \label{km-1} $\mu_G(\varnothing) = 0$ and $\mu_G(G) = 1$.
  \item \label{km-2} $0 \le \mu_G(D) \le 1$.
  \item \label{km-3} $\mu_G(D \cup D') = \mu_G(D) + \mu_G(D') - \mu_G(D \cap D')$.
  \item \label{km-4} If $a \in G$, then $\mu_G(a \cdot D) = \mu_G(D)$.
  \item \label{magic} For any $\mathcal{L}_{rings}$-formula $\phi(x;y)$ of
    VC-dimension $k$ and any $n$, if $N = N_{k,1/n}$ is as in
    Remark~\ref{notation}, then there are $a_1,\ldots, a_N \in G$ such that
    if $D = \phi(\Mm;b) \subseteq G$ and $\mu_G(D) > 1/n$, then $D \cap
    \{a_1,\ldots,a_N\} \ne \varnothing$.
  \end{enumerate}
\end{proposition}
\begin{proof}
  The proof is similar to Lemma~\ref{ridiculemma}, transfering things
  from $\Qq_p^\dag$ to $\Mm^\dag$.  Part (\ref{magic}) uses
  Proposition~\ref{main-target}.
\end{proof}
\begin{proposition}\label{hard-dir}
  Let $G$ be a definably compact group in $\Mm$.  Then $G$ has
  \textit{fsg}.
\end{proposition}
\begin{proof}
  For any definable set $D \subseteq G$, let $\mu^*(D)$ be the
  standard part of $\mu_G(D)$.  Parts (\ref{km-1})--(\ref{km-4}) of
  Proposition~\ref{keisler-measure} imply that $\mu^*(D)$ is a
  left-invariant Keisler measure on $G$.  By part (\ref{magic}), we
  can find a countable set $A \subseteq G$ such that if $\mu^*(D) >
  0$, then $D \cap A \ne \varnothing$.  Then $\mu^*$ and all of its left
  translates are finitely satisfiable in a small model.  By
  \cite[Remark~4.6]{goo}, $G$ has \textit{fsg}.
\end{proof}
The $\Mm$ of this section is a monster model of the complete theory
$p$CF, so Proposition~\ref{hard-dir} applies to definable groups in
\emph{any} model of $p$CF.  Combined with Proposition~\ref{easy-dir},
Theorem~\ref{main-thm} follows.

\begin{acknowledgment}
  The author was supported by the National Natural Science Foundation
  of China (Grant No.\@ 12101131).  This paper grew out of discussions
  with Ningyuan Yao.  Anand Pillay provided some helpful references.
\end{acknowledgment}
\bibliographystyle{plain} \bibliography{mybib}{}

\end{document}